\documentclass[a4paper,english,fontsize=10pt,parskip=half,abstracton]{scrartcl}
\usepackage{babel}
\usepackage[utf8]{inputenc}
\usepackage[T1]{fontenc}
\usepackage[a4paper,left=20mm,right=20mm,top=30mm,bottom=30mm]{geometry}
\usepackage{amsmath}
\usepackage{amsthm}
\usepackage{amssymb}
\usepackage{enumerate}
\newtheorem{Lemma}{Lemma}
\newtheorem{Theorem}{Theorem}
\newtheorem{Proposition}{Proposition}

\renewcommand{\phi}{\varphi}
\newcommand{\C}{\operatorname{C}}
\newcommand{\N}{\operatorname{N}}
\newcommand{\Aut}{\operatorname{Aut}}
\newcommand{\pcore}{\operatorname{O}}

\mathchardef\ordinarycolon\mathcode`\:  %defines a nice ":=" 
 \mathcode`\:=\string"8000
 \begingroup \catcode`\:=\active
   \gdef:{\mathrel{\mathop\ordinarycolon}}
 \endgroup

\title{Fusion systems on metacyclic $2$-groups}
\author{
Benjamin Sambale\\
Mathematisches Institut\\
Friedrich-Schiller-Universität\\
07743 Jena\\
Germany\\
{\tt benjamin.sambale@uni-jena.de}
}
\date{\today}

\begin{document}
\frenchspacing
\maketitle
\begin{abstract}
Let $P$ be a finite metacyclic $2$-group and $\mathcal{F}$ a fusion system on $P$. We prove that $\mathcal{F}$ is nilpotent unless $P$ has maximal class or $P$ is homocyclic, i.\,e. $P$ is a direct product of two isomorphic cyclic groups. As a consequence we obtain the numerical invariants for $2$-blocks with metacyclic defect groups. This paper is a part of the author's PhD thesis.
\end{abstract}

\section{Introduction}
The structure of a finite group $G$ is often controlled by the structure of its Sylow $p$-subgroups and the way they are embedded into $G$. In this paper, we give an example of such a control in the case $p=2$. More precisely, we show that $G$ is $2$-nilpotent if the Sylow $2$-subgroups of $G$ are metacyclic, but neither of maximal class nor homocyclic. Here, a homocyclic group means a direct product of two isomorphic cyclic groups. Moreover, the $2$-groups of maximal class are precisely the dihedral groups, the semidihedral groups, and the quaternion groups. Of course, there are many other metacyclic $2$-groups (see for example \cite{Liedahl}). In this sense, most of the metacyclic $2$-groups satisfy our assumption. In the easiest case where the Sylow $2$-subgroups of $G$ are cyclic, the result is a well-known theorem by Burnside. Another case, in which the result is known, is due to Wong. He showed that $G$ is $2$-nilpotent if the Sylow $2$-subgroups of $G$ are the sometimes called modular groups (see Satz~IV.3.5 in \cite{Huppert}). These are the nonabelian $2$-groups with cyclic subgroups of index $2$, which do not have maximal class. For similar results see \cite{pnilpotentforcing}. 

For the proof, we use the notion of fusion systems, which provide a great generalization of the situation described above. In particular, our result applies also to $2$-blocks of finite groups with metacyclic defect groups. This was already observed by Robinson without the notion of fusion systems (see \cite{Robinsonmetac}). As a consequence we state the numerical invariants of such blocks.
We note that David Craven and Adam Glesser also found the main result of this paper independently, but they did not publish it yet. Furthermore, Radu Stancu has studied fusion systems on metacyclic $p$-groups for odd primes $p$ (see \cite{Stancu}).

\section{Metacyclic $2$-groups}
First, we collect some elementary results about metacyclic $2$-groups. We will freely use the fact that subgroups and quotients of metacyclic groups are again metacyclic. 

\begin{Lemma}\label{autmetac}
If $P$ is a metacyclic $2$-group, then the automorphism group $\Aut(P)$ is also a $2$-group unless $P$ is homocyclic or a quaternion group of order $8$.
\end{Lemma}

\begin{proof}
See Lemma~1 in \cite{Mazurov}.
\end{proof}

It is easy to show that the converse of Lemma~\ref{autmetac} also holds. In the following, we denote the cyclic group of order $n\in\mathbb{N}$ by $C_n$. Moreover, we set $C_n^2:=C_n\times C_n$. For a finite group $G$, we denote by $\Phi(G)$ the Frattini subgroup of $G$. If $H$ is a subgroup of $G$, then $\N_G(H)$ is the normalizer and $\C_G(H)$ is the centralizer of $H$ in $G$.

\begin{Lemma}\label{action}
Let $P$ be a metacyclic $2$-group and $C_{2^k}^2\cong Q<P$ with $k\ge 2$. Then the action of $\N_P(Q)/Q$ on $Q/\Phi(Q)$ by conjugation is not faithful.
\end{Lemma}

\begin{proof}
By way of contradiction, we assume that $\N_P(Q)/Q$ acts faithfully on $Q/\Phi(Q)$. Hence $|\N_P(Q):Q|=2$, since $|\Aut(Q/\Phi(Q))|=6$. For simplicity, we may also assume $P=\N_P(Q)$.
Then $P/\Phi(Q)$ is a dihedral group of order $8$. Let $\langle x\rangle\unlhd P$ such that $P/\langle x\rangle$ is cyclic. Then $\langle x\rangle\Phi(Q)/\Phi(Q)$ and $(P/\Phi(Q))/(\langle x\rangle\Phi(Q)/\Phi(Q))\cong P/\langle x\rangle\Phi(Q)$ are also cyclic. Since a dihedral group of order $8$ cannot have a cyclic quotient of higher order than $2$, this shows $|P/\langle x\rangle\Phi(Q)|=2$. Since $Q/\Phi(Q)$ is noncyclic, $x\notin Q$.
% and $|\langle x\rangle|=4|\langle x\rangle\cap\Phi(Q)|$. Because $x^2\in Q$, $x$ has at most order $2^{k+1}$.
%If $x$ has lower order than $2^{k+1}$, then $x^2$ would be contained in $\Phi(Q)$. Thus $|\langle x\rangle|=2^{k+1}$. In particular $x\notin Q$. 
The restriction map $\phi:\Aut(\langle x\rangle)\to\Aut(\langle x^2\rangle)$ is an epimorphism with kernel of order $2$. The action of $Q$ on $\langle x\rangle$ induces a homomorphism $\psi:Q\to\Aut(\langle x\rangle)$ with image contained in $\ker(\phi)$, because $x^2\in Q$. Since $\ker(\phi)$ has order $2$, $\Phi(Q)$ is contained in $\ker(\psi)$. In particular $x$ centralizes $\Phi(Q)$. By Burnside's basis theorem, there exists an element $y\in Q$ such that $Q=\langle y,xyx^{-1}\rangle$. Therefore $\Phi(Q)=\langle y^2,xy^2x^{-1}\rangle=\langle y^2\rangle$ is cyclic. But this contradicts $k\ge 2$.
\end{proof}

In the situation of Lemma~\ref{action} one can also show that $Q$ is the only subgroup of type $C_{2^k}^2$ for a fixed $k$. In particular $\N_P(Q)=P$. However, we will not need this in the following.

\section{Fusion systems}
In this section, we will use the definitions and results of \cite{Linckelmann} to prove the main result of this paper. Moreover, we say that a fusion system $\mathcal{F}$ on a finite $p$-group $P$ is \emph{nilpotent} if $\Aut_{\mathcal{F}}(Q)$ is a $p$-group for every $Q\le P$ (cf. Theorem~3.11 in \cite{Linckelmann}). 

\begin{Theorem}\label{mainthm}
Let $P$ be a metacyclic $2$-group, which is neither of maximal class nor homocyclic. Then every fusion system $\mathcal{F}$ on $P$ is nilpotent.
\end{Theorem}

\begin{proof}
In order to show that $P$ does not contain $\mathcal{F}$-essential subgroups, let $Q\le P$ be $\mathcal{F}$-essential. Since $\Aut_{\mathcal{F}}(Q)/\Aut_Q(Q)$ has a strongly $2$-embedded subgroup, $\Aut(Q)$ is not a $2$-group. By the Lemma~\ref{autmetac},
$Q$ is homocyclic or a quaternion group of order $8$. By Proposition~10.19 in \cite{Berkovich1}, $Q$ cannot be a quaternion group.
Thus, we may assume $Q\cong C_{2^k}^2$ for some $k\in\mathbb{N}$. Moreover, $\C_P(Q)=Q$ holds, because $Q$ is also $\mathcal{F}$-centric. First, we consider the case $k=1$. Then $Q\subseteq\langle x\in P:x^2=1\rangle=\Omega(P)$, and Exercise~1.85 in \cite{Berkovich1} implies $Q=\Omega(P)\unlhd P$. Since $P/Q=\N_P(Q)/\C_P(Q)$ is isomorphic to a subgroup of $\Aut(Q)$, $P$ has order $8$. This contradicts our hypothesis. Hence, we have $k\ge 2$. Now consider the action of $\Aut_{\mathcal{F}}(Q)=\Aut_{\mathcal{F}}(Q)/\Aut_Q(Q)$ on $Q/\Phi(Q)$. Lemma~\ref{action} shows that $1\ne\Aut_P(Q)\subseteq\Aut_{\mathcal{F}}(Q)$ does not act faithfully. On the other hand, it is well-known that every automorphism of $Q$ of odd order acts nontrivially on $Q/\Phi(Q)$. Therefore the kernel of the action under consideration forms a nontrivial normal $2$-subgroup of $\Aut_{\mathcal{F}}(Q)$, i.\,e. $\pcore_2(\Aut_{\mathcal{F}}(Q))\ne 1$. But this contradicts the fact that $\Aut_{\mathcal{F}}(Q)$ contains a strongly $2$-embedded subgroup.

Now let $Q$ be an arbitrary subgroup of $P$. We have to show that $\Aut_{\mathcal{F}}(Q)$ is a $2$-group. Let $\phi\in\Aut_{\mathcal{F}}(Q)$. Then Alperin's fusion theorem (Theorem~5.2 in \cite{Linckelmann}) shows that $\phi$ is the restriction of an automorphism of $P$. But again by Lemma~\ref{autmetac}, $\Aut(P)$ is a $2$-group, and $\phi$ must be a $2$-element. This proves the theorem.
\end{proof}

The next statement is in some sense a converse of Theorem~\ref{mainthm}.

\begin{Proposition}
Let $P$ be a $2$-group of maximal class or a homocyclic $2$-group. Then there exists a fusion system on $P$, which is not nilpotent.
\end{Proposition}

\begin{proof}
It suffices to show that there exists a finite group $G$ with $P$ as a Sylow $2$-subgroup such that $G$ is not $2$-nilpotent.
If $P$ is homocyclic, then the claim follows from Theorem~1.10 in \cite{pnilpotentforcing}. If $P$ has maximal class, then Theorem~2.4, Theorem~2.5, and Theorem~2.6 in \cite{pnilpotentforcing} imply the result.
\end{proof}

\section{$2$-blocks of finite groups}
Now we turn to blocks. Let $G$ be a finite group, and let $B$ be a $2$-block of $G$. We denote the number of irreducible ordinary (modular) characters of $B$ by $k(B)$ ($l(B)$) respectively. Further, we define $k_i(B)$ as the number of irreducible ordinary characters of height $i\in\mathbb{N}_0$. It is well known that the so called subpairs for $B$ provide a fusion system on a defect group $D$ of $B$. Let us assume that $D$ is metacyclic. If $D$ has maximal class, the numbers $k(B)$, $k_i(B)$ and $l(B)$ were obtained by Brauer and Olsson (see \cite{Brauer,Olsson}). In the case $D\cong C_{2^n}^2$ for some $n\in\mathbb{N}$ the inertial index $e(B)$ of $B$ is $1$ or $3$. For $e(B)=1$ the fusion system (and the block) has to be nilpotent. Thus, we may assume $e(B)=3$. Then Usami and Puig state (without an explicit proof) that there exists a perfect isometry between $B$ and the group algebra of $D\rtimes C_3$ (see \cite{Usami23I,Usami23II,UsamiZ4}). The author verified this as a part of his PhD thesis. In particular the numbers $k(B)$, $k_i(B)$ and $l(B)$ can be calculated easily. 
In all other cases for $D$ the block $B$ has to be nilpotent be Theorem~\ref{mainthm}. In this case a theorem by Puig applies (see \cite{Puig}). We summarize all these results:

\begin{Theorem}\label{2blocks}
Let $B$ be a $2$-block of a finite group $G$ with a metacyclic defect group $D$. Then one of the following holds:
\begin{enumerate}[(1)]
\item $B$ is nilpotent. Then $k_i(B)$ is the number of ordinary characters of $D$ of degree $2^i$. In particular $k(B)$ is the number of conjugacy classes of $D$ and $k_0(B)=|D:D'|$. Moreover, $l(B)=1$.
\item $D$ is a dihedral group of order $2^n\ge 8$. Then $k(B)=2^{n-2}+3$, $k_0(B)=4$ and $k_1(B)=2^{n-2}-1$. According to two different fusion systems, $l(B)$ is $2$ or $3$.
\item $D$ is a quaternion group of order $8$. Then $k(B)=7$, $k_0(B)=4$ and $k_1(B)=l(B)=3$.
\item $D$ is a quaternion group of order $2^n\ge 16$. Then $k_0(B)=4$ and $k_1(B)=2^{n-2}-1$. According to two different fusion systems, one of the following holds
\begin{enumerate}[(a)]
\item $k(B)=2^{n-2}+4$, $k_{n-2}(B)=1$ and $l(B)=2$.
\item $k(B)=2^{n-2}+5$, $k_{n-2}(B)=2$ and $l(B)=3$. 
\end{enumerate}
\item $D$ is a semidihedral group of order $2^n\ge 16$. Then $k_0(B)=4$ and $k_1(B)=2^{n-2}-1$.
According to three different fusion systems, one of the following holds
\begin{enumerate}[(a)]
\item $k(B)=2^{n-2}+3$ and $l(B)=2$.
\item $k(B)=2^{n-2}+4$, $k_{n-2}(B)=1$ and $l(B)=2$.
\item $k(B)=2^{n-2}+4$, $k_{n-2}(B)=1$ and $l(B)=3$.
\end{enumerate}
\item $D$ is homocyclic. Then $k(B)=k_0(B)=\frac{|D|+8}{3}$ and $l(B)=3$.
\end{enumerate}
\end{Theorem}

\end{document}